\documentclass[reqno]{amsart}
\usepackage{amsmath}
\usepackage{mathrsfs}

\usepackage{mathpazo}

\usepackage{url}

\makeatletter \@addtoreset{equation}{section} \makeatother

\renewcommand\thefigure{\thesection.\@arabic\c@figure}
\renewcommand\thetable{\thesection.\@arabic\c@table}
\newtheorem{theorem}{Theorem}[section]
\newtheorem{lemma}[theorem]{Lemma}
\newtheorem{proposition}[theorem]{Proposition}
\newtheorem{corollary}[theorem]{Corollary}

\newtheorem{definition}[theorem]{Definition}

\theoremstyle{remark}

\newtheorem{remark}[theorem]{Remark}

\newcommand{\mc}[1]{{\mathscr #1}}
\newcommand{\mf}[1]{{\mathfrak #1}}

\newcommand{\bb}[1]{{\mathbb #1}}

\newcommand{\<}{\langle}
\renewcommand{\>}{\rangle}

\usepackage[]{epsfig}
\usepackage[]{pstricks}

\newcommand{\eps}{\varepsilon}

\renewcommand{\>}{\rangle}

\newcommand{\pfrac}[2]{\genfrac{}{}{}{1}{#1}{#2}}

\newcommand{\at}[2]{\genfrac{}{}{0pt}{}{#1}{#2}}

\newcommand{\dl}{\<\!\<}
\newcommand{\dr}{\>\!\>}


\date{}

\begin{document}

\title[Occupation times of exclusion processes with conductances]{Occupation times of exclusion processes\\ with conductances}

\author{Tertuliano Franco}
\address{UFBA\\
 Instituto de Matem\'atica, Campus de Ondina, Av. Adhemar de Barros, S/N. CEP 40170-110\\
Salvador, Brasil}
\curraddr{}
\email{tertu@impa.br}

\author{Patr\'{\i}cia Gon\c{c}alves}
\address{ PUC-RIO, Departamento de Matem\'atica, Rua Marqu\^es de S\~ao Vicente, no. 225, 22453-900, G\'avea, Rio de Janeiro, Brazil \\and\\
CMAT, Centro de Matem\'atica da Universidade do Minho, Campus de Gualtar, 4710-057 Braga, Portugal.}
\email{patricia@mat.puc-rio.br}

\author{Adriana Neumann}
\address{UFRGS, Instituto de Matem\'atica, Campus do Vale, Av. Bento Gon\c calves, 9500. CEP 91509-900, Porto Alegre, Brasil}
\curraddr{}
\email{aneumann@impa.br}

\begin{abstract}
We obtain the fluctuations for the occupation time of one-dimensional symmetric exclusion processes with speed change, where the transition rates ({\em conductances}) are driven by a general function $W$. The approach does not require sharp bounds on the spectral gap of the system nor the jump rates to be bounded from above or below.  We present some examples and for one of them, we observe that the fluctuations of the current are trivial, but the fluctuations of the occupation time are given by a fractional Brownian Motion. This shows that, in general, the fluctuations of the current and of the occupation time are not of same order.
\end{abstract}

\maketitle

\section{Introduction}

\emph{Occupation time} is the usual nomenclature for the additive functional $\int_0^t\eta_s(x)ds$, where $\eta_s(x)$ denotes the occupation variable at the site $x$ at the time $s$. Namely, $\eta_s(x)$ represents how many particles stand at the site $x$ and at the time $s$ for some particle system $\{\eta_t\,: \,t\geq 0\}$. In this paper, we are concerned with a standard interacting particle system, the \emph{exclusion process}. Succinctly, the exclusion process consists in a system of random walks evolving on a lattice under the rule that a particle can not jump to an already occupied site. This is the so-called \emph{exclusion rule}. Such model is of great importance in Probability and Statistical Mechanics for several reasons. At the same time it has a simple interaction among particles but its peculiarities allow to prove deep results  which are shared by many other models.

We consider here one-dimensional speed change exclusion processes. The dynamics of these process can be informally described as follows. To each bond of the lattice it is associated a Poisson clock, whose parameter is a function of the position of the bond \cite{fgn,fgn2,fgn3,fl}. The system is taken to start from the equilibrium state, which consists in a Bernoulli product measure with constant parameter.
Our main result is the derivation of a functional central limit for the occupation time, when suitably re-scaled.

There is a vast literature on the fluctuations of the occupation time of symmetric particle systems, see for instance \cite{GJ,S,SX} and references therein. In this paper we follow the approach proposed in \cite{GJ}, which consists in replacing the occupation time functional by an additive functional of the density of particles. Then, as a consequence of the Central Limit Theorem for the density of particles, we deduce the corresponding result for the occupation time functional. We consider exclusion processes with speed change for which the Central Limit Theorem for the density of particles has been derived \cite{fsv}. Therefore, to complete our goals we just need to justify the proper replacement of the aforementioned functionals. For that purpose, we introduce what we call a {\em Local Replacement} which allows to substitute the occupation time functional by an additive functional of the empirical average of particles on a small macroscopic box. This Local Replacement avoids performing a multi-scale analysis in
order to derive a second order Boltzmann Gibbs Principle as in  \cite{GJ}. More than that, we do not require sharp bounds on the spectral gap, nor the boundedness of the jump rates of the system, as required in \cite{GJ}. Therefore, our results are true for a general class of exclusion processes, for which the methods of \cite{GJ} do not apply directly. On the other hand,  our results are not as general as the results of \cite{GJ}, since they only hold for the occupation time functional and no other additive functional. We believe that our method can be extended to more general dynamics than the exclusion constraint, but this is left for future work.

We present here some particular cases of interest. First, we consider porous media models which were analyzed in \cite{glt} and correspond to taking $W$ as the identity function. These models do not satisfy the spectral gap bound required in \cite{GJ} but with our method we obtain the fractional Brownian Motion ruling the fluctuations of the occupation time. Second, we consider \emph{exclusion processes with a slow bond} which were analyzed in \cite{fgn3}. These models do not satisfy the boundedness of the jump rates as required in \cite{GJ}, but our method also fits these models. We remark that \emph{exclusion processes with a slow bond} is an interesting example of a particle system for which the fluctuations of the current and the fluctuations of the occupation time have completely different behaviors. This shows that, in general, the fluctuations of the current and of the occupation time are not of same type.

This paper has the following outline.
In Section 2, we define our models and we state our main result, namely Theorem \ref{th2.1}. In Section 3, we recall the hydrodynamic limit and the fluctuations of the  density from \cite{fl} and \cite{fsv}, respectively. In Section 4, we prove our main result. Section 5 is devoted to examples: porous media models and exclusion processes with a slow bond. In the Appendix we present some technical lemmas.

\section{The main result}\label{s2}
Denote by $\bb T = \bb R/\bb Z = [0, 1)$ the one-dimensional continuous torus, and by $\bb T_n = \bb Z/n\bb Z =\{0,\ldots, n - 1\}$ the one-dimensional discrete torus with $n$ points.

Fix $W : \bb R\to \bb R$ a strictly increasing right continuous function with left limits
(c\`adl\`ag), periodic in the sense that, for all $u\in \bb R$,
\begin{equation}\label{periodic}
W(u + 1) - W(u) = W(1) - W(0).
\end{equation}
Consider the state space $\Omega_n:=\{0,1\}^{\mathbb{T}_n}$.
The \emph{speed change exclusion process with conductances} is the Markov process $\{\eta_t:t\geq{0}\}$ whose infinitesimal generator acts on local functions $f:\Omega_n\rightarrow{\mathbb{R}}$ as
\begin{equation}\label{ln}
(\mc{L}_{n}f)(\eta)=\sum_{x\in \bb T_n}\,\xi^{n}_{x,x+1}\,c_{x,x+1}(\eta)\,[f(\eta^{x,x+1})-f(\eta)]\,,
\end{equation}
where
 $\eta^{x,x+1}$ is the configuration obtained from $\eta$ by exchanging the occupation variables $\eta(x)$ and $\eta(x+1)$:
\begin{equation}\label{exchange}
(\eta^{x,x+1})(y)=\left\{\begin{array}{cl}
\eta(x+1),& \mbox{if}\,\,\, y=x\,,\\
\eta(x),& \mbox{if} \,\,\,y=x+1\,,\\
\eta(y),& \mbox{otherwise,}
\end{array}
\right.
\end{equation}
the conductantes $\xi^n_{x,x+1}$ are given by
\begin{equation*}
\xi^n_{x,x+1}\;=\;\frac{1}{n\big(W\big(\pfrac{x+1}{n}\big)-W\big(\pfrac{x}{n}\big)\big)}
\end{equation*}
and
\begin{equation*}
c_{x,x+1}(\eta)\;=\;1+b(\eta(x-1)+\eta(x+2)),
\end{equation*}
with $b>-1/2$.

Throughout this paper, we assume the following technical condition on the function $W$: there exists a constant $\theta>0$ such that
\begin{equation}\label{Wassumption}
\frac{1}{\varepsilon n}\sum_{y=0}^{\varepsilon n-1}\big(W\big(\pfrac{y}{n}\big)-W(0)\big)\sim O(\varepsilon^{\theta})\,,
\end{equation}
where  $f \sim O(g)$ means that the function $f$ is bounded from above by a constant times the function $g$.

To exemplify the assumption \eqref{Wassumption}, if $W$ is a $\theta$-H\"older function in a neighborhood of zero, then \eqref{Wassumption} is satisfied, since
\begin{equation*}
\frac{1}{\varepsilon n}\sum_{y=0}^{\varepsilon n-1}\big(W\big(\pfrac{y}{n}\big)-W(0)\big)\leq\frac{C_W}{\varepsilon n}\sum_{y=0}^{\varepsilon n-1}\frac{y^\theta}{n^\theta}\leq\frac{C_W}{\varepsilon n}\sum_{y=0}^{\varepsilon n-1}\varepsilon^\theta=C_W\varepsilon^\theta\,,
\end{equation*}
where $C_W$ is the H\"older constant.

 The dynamics of the process $\{\eta_t:t\geq 0\}$ can be informally described as follows. At each bond $\{x,x+1\}$ of $\bb T_n$, there is an exponential clock of parameter $\xi^n_{x,x+1}$, all of them being independent. Suppose the configuration at the present is $\eta$. After a ring of the clock at the  bond $\{x,x+1\}$, the occupation variables $\eta(x)$ and $\eta(x+1)$ are exchanged at rate $c_{x,x+1}(\eta)$. 

We remark that the condition $b>-1/2$ is required to ensure that the system is ergodic in the following sense. First, we notice that the dynamics introduced above conserves the total number of particles. Therefore, the state space of the process can be written as $\Omega_n:=\bigcup_{k=0}^n\mc{H}_{n,k}$, where $\mc{H}_{n,k}$ denotes the hyperplane of configurations in $\Omega_n$ with $k$ particles. The ergodicity property means that on each hyperplane, with positive probability, we can reach any configuration in the same hyperplane using the allowed jumps of the dynamics. For instance, if $b=-1/2$ and for a configuration $\eta$ having the sites $x-1$, $x$,  $x+2$ occupied, and the site $x+1$ empty, then $c_{x,x+1}(\eta)=1+2b=0$. Then, for this choice of $b$ there are blocked configurations, that is, configurations that do not evolve under the dynamics. Therefore, the system in not ergodic.

Also, it is well known that the Bernoulli product measures on $\Omega_n$ with parameter $\rho\in{[0,1]}$, denoted by
$\{\nu_\rho : 0\le \rho \le 1\}$, are invariant for the dynamics introduced above.   Moreover,  they  are also
reversible.

 Fix $T>0$. The trajectories of $\{\eta_t : t\ge 0\}$ live on the space $\mc D([0,T], \Omega_n)$, that is, the path space of
c\`adl\`ag trajectories with values in $\Omega_n$. For a
measure $\nu_\rho$ on $\Omega_n$, we denote by $\bb P_{\nu_\rho}$ the
probability measure on $\mc D([0,T], \Omega_n)$ induced by $\nu_\rho$ and by $\{\eta_t : t\ge 0\}$ and we denote by  $\bb E_{\rho}$
the expectation with respect to $\bb P_{\nu_\rho}$.

Let $\mc{C}([0,T],\mathbb{R})$  be the path space of continuous trajectories with values in $\bb R$.

 \begin{definition}
The occupation time of the origin is defined as the additive functional
\begin{equation}\label{ot}
\Gamma_n(t):=\frac{1}{n^{3/2}}\int_{0}^{tn^2}(\eta_{s}(0)-\rho)\,ds .
\end{equation}
\end{definition}
\noindent The definition above has already the correct scaling in terms of $n$, in order to $\Gamma_n(t)$ have a non trivial limit when taking $n$ to infinity.
The occupation time at a site $x\in\mathbb{T}_n$ is defined as above by replacing $\eta_{s}(0)$ by $\eta_{s}(x)$.

Our main result is the following
\begin{theorem}\label{th2.1}(Fluctuations of the occupation time)\\
As $n$ goes to infinity, the sequence of processes $\{\Gamma_n(t):t\in{[0,T]}\}_{n\in{\mathbb{N}}}$ converges in distribution, with respect to the uniform topology of $\mc{C}([0,T],\mathbb{R})$, to a Gaussian process $\{\Gamma(t): t\in{[0,T]}\}$.
\end{theorem}
\begin{remark}
We notice that the previous result also holds for the occupation time of any site $x\in{\mathbb{T}_n}$, by replacing condition \eqref{Wassumption} by
\begin{equation}\label{Wassumption1}
\frac{1}{\varepsilon n}\sum_{y=x}^{x+\varepsilon n-1}\big(W\big(\pfrac{y}{n}\big)-W\big(\pfrac{x}{n}\big)\big)\sim O(\varepsilon^{\theta}).
\end{equation}
For ease of notation we opt to present the result for $x=0$.
\end{remark}

\section{Scaling limits: hydrodynamics and fluctuations}
In this section we review the hydrodynamic limit and the equilibrium fluctuations of the density, for the models introduced above.

\subsection{Hydrodynamic Limit.}

In words, the hydrodynamic limit consists in the analysis of the time evolution of the spatial density of particles. This spatial density of particles is represented by the empirical measure process
$\pi^{n}_t(\eta,du):=\pi^n(\eta_t,du)$
defined, for $t\in{[0,T]}$, by
\begin{equation*}
\pi^{n}(\eta_t,du) \;=\; \pfrac{1}{n} \sum _{x\in \bb T_n} \eta_{tn^2} (x)
\delta_{\frac{x}{n}}(du)\;\in\; \mc M,
\end{equation*}
where $\delta_y$ is the Dirac measure concentrated on $y\in \bb T$. Above, $\mc M$ denotes the space of positive measures on $\bb T$ with total
mass bounded by one, endowed with the weak topology.
To  uniquely characterize the time evolution of the empirical measure, some condition must be imposed on the starting measures. This is the content of next definition.

\begin{definition} \label{def associated measures}
A sequence of probability measures $\{\mu_n  \}_{n\in \bb N}$, where $\mu_n$ is a probability  measure on $\Omega_n$, is
said to be associated to a profile $\psi_0 :\bb T \to [0,1]$ if for every $\delta>0$ and every continuous function $H: \bb T \to \bb R$
\begin{equation}\label{associated}
\lim_{n\to\infty}
\mu_n \Big\{ \eta\in \Omega_n\;:\; \Big\vert \pfrac 1n \sum_{x\in\bb T_n} H(\pfrac{x}{n})\, \eta(x)
- \int_{\bb{T}} H(u)\, \psi_0(u) du \Big\vert > \delta \Big\} \;=\; 0.
\end{equation}
\end{definition}

In \cite{fl} it was proved that:

 \begin{theorem}\label{t1}

Fix a continuous profile $\psi_0 : \bb T \to [0,1]$.  Let $\{\mu_n\}_{n\geq{1}}$ be
a sequence of probability measures   associated to $\psi_0$.
 Then, for any $t\in [0,T]$, for every $\delta>0$ and every continuous function $H: \bb T \to \bb R$, it holds that
\begin{equation*}
\lim_{n\to\infty}
\bb P_{\mu_n} \Big\{\eta_. : \, \Big\vert \pfrac{1}{n} \sum_{x\in\bb{T}_n}
H(\pfrac{x}{n})\, \eta_t(x) - \int_{\bb T} H(u)\, \psi(t,u) du \Big\vert
> \delta \Big\} \;=\; 0\,,
\end{equation*}
 where $\psi : [0,T] \times \bb T \to \bb R$
 is the unique weak solution of
\begin{equation}\label{edp2}
\left\{
\begin{array}{l}
{\displaystyle \partial_t \psi\; =\; \mf L_W \psi }\,, \\
{\displaystyle \psi(0,u) \;=\; \psi_0(u)}\,, \forall u\in{\mathbb{T}_n}\,.
\end{array}
\right.
\end{equation}
The operator $\mf L_W $ is defined in next subsection, as well the notion of weak solution of \eqref{edp2}.
\end{theorem}
In order to state properly what is a weak solution of \eqref{edp2} we need to introduce some definitions.

\subsection{The operator $ \mf L_W$}

We detail here the operator $\mf L_W: \mf D_W\subset L^2(\bb T)\to L^2(\bb T)$. We start by defining its domain $\mf D_W$.
For that purpose, we consider $W(dy)$
 as the measure on the continuous torus associated to the function $W:\bb R\to \bb R$ in the usual way, or else, as the unique measure such that
 \begin{equation}
W((a,b]):=W(b)-W(a)\,\quad \forall \,a, b\in \bb T\, \textrm{with}\,a<b.
\end{equation}
Notice that the periodicity condition given in \eqref{periodic} guarantees that the measure above is well defined.

  The domain $\mf D_W$ consists on the set of functions $G$ in $L^2(\bb T)$ such that
\begin{equation*}
 G(u)=a+b\,W(u)+\int_{(0,u]}\Big(\int_0^y g(z)\,dz\Big)W(dy),\quad\forall \, u\in \bb T,
\end{equation*}
for some function $g$ in $L^2(\bb T)$ that satisfies
\begin{equation*}
 \int_0^1 g(z)\,dz=0\qquad\mbox{and}\qquad \int_{(0,1]}\Big(b+\int_0^y g(z)\,dz\Big)W(dy)=0.
\end{equation*}

 The operator $ \mf L_W$ acts on $G\in\mf D_W$
as $\mf L_W G=g$. An alternative definition of the operator can be stated in the following way. Denote by $\partial_u$ the usual space derivative and define
the generalized derivative $\partial_W$ of a function $G:\bb T\to \bb R$ by
\begin{equation}
 \partial_W G(u)=\lim_{\eps\to 0}\frac{G(u+\eps)-G(u)}{W(u+\eps)-W(u)},
\end{equation}
when the above limit exists and is finite. Keeping this in mind, given $G\in\mf D_W$,  we have  $\mf L_W G(u)=\partial_u\partial_W G(u)$, for all $u\in \bb T$.

This operator $\mf L_W$ is a Krein-Feller type operator (see e.g. \cite{f} on the subject). In \cite{fl}, it was proved that $\mf L_W$ satisfies the  properties stated in the ensuing theorem. Below $\dl \cdot, \cdot\dr$ denotes the inner product in $L ^2(\bb T)$ and $\Vert\cdot \Vert$ the corresponding norm.
\begin{theorem}\label{thm_fl}
There exists an Hilbert space $\mc H^1_W$  compactly embedded in  $L^2(\bb T)$
such that $\mf D_W \subset \mc H^1_W$ and $\mf L_W$ can be extended to $\mc H^1_W$ such that the extension enjoys the following properties:
\begin{itemize}
\item[(a)] The domain $\mc H^1_W $ is dense in $L^2(\bb T)$;
\item[(b)] The operator $\mf L_W$ is self-adjoint and non-positive $\dl H, -\mf L_W H\dr \geq 0,$ for all $H\in\mc H^1_W $;
\item[(c)] Let $\bb I$ be the identity operator. The operator $\bb I -\,\mf L_W : \mc H^1_W \rightarrow L^2(\bb T)$ is bijective and $\mf D_W$ is a core of it;
\item[(d)] The operator $\mf L_W$ is dissipative, i.e., $\Vert \mu H -\mf L_W H\Vert \geq \mu \Vert H\Vert \, ,$
for some $\mu >0$ and for all $H\in\mc H^1_W $;
\item[(e)] The eigenvalues of $- \mf L_W$ form a countable set
$0= \mu_0\le \mu_1\le \mu_2\le \cdots$ with $\lim_{n\to\infty} \mu_n  = \infty$, and all of them have finite multiplicity;
\item[(f)] There exists a complete orthonormal basis of $L^2(\bb T)$ composed of eigenfunctions $\varphi_n$ of $- \mf L_W$ associated to the eigenvalues $\mu_n$.
\end{itemize}
In view of (a), (c) and (d), by the Hille-Yosida Theorem, $\mf L_W$ is the generator of a strongly continuous contraction semigroup in $L^2(\bb T)$.
\end{theorem}


Finally, we state what is meant to be a weak solution to \eqref{edp2}.

\begin{definition}\label{def weak solution edp2}
A bounded function $\psi : [0,T] \times \bb T \to \bb R$
 is said to be a weak solution of the parabolic differential equation \eqref{edp2}
if, for any $t\in{[0,T]}$ and any $H\in\mc H^1_W$, the function $\psi(t,\cdot)$ satisfies the integral equation
\begin{equation*}
\int_{\bb T} \psi(t,u) H(u)\,du\;-\; \int_{\bb T} \psi(0,u) H(u)\,du\;
 -\;\int_0^t \int_{\bb T} \psi(s,u) \mf L_W \,H(u)\,du\, ds\;=\;0.
\end{equation*}
\end{definition}

\subsection{Equilibrium fluctuations and the generalized Ornstein-Uhlenbeck process}\label{sub3.3}

Following the ideas of \cite{fsv}, we define $S_W(\bb T)=\bigcap_{n=0}^{\infty}S_n$, where $S_n$ is the Hilbert space obtained by completing the space $\mf{D}_W$  with respect to the inner product $\<\cdot,\cdot\>_n$ given by
\begin{equation}\label{inner_product}
\<f,g\>_n=\sum_{k=1}^{\infty}(1+\mu_{k})^{2n}k^{2n}\int_{\bb T}P_kf(u)P_kg(u)du,
\end{equation}
where $P_k$ is the orthogonal projection on the linear space generated by the eigenfunction $\varphi_k$ given in Theorem \ref{thm_fl}. Let $S_W'(\bb T)$ denote the dual space of $S_W(\bb T)$, that is, the space of the bounded linear functionals from $S_W(\bb T)$ to $\mathbb{R}$.

We define the density fluctuation field, which is an element of $S'_W(\bb T)$, as the linear functional acting on functions $H\in{{S_W(\bb T)}}$
as
\begin{equation}\label{fluctfield}
\mc{Y}^n_t(H)=\frac{1}{\sqrt n}\sum_{x\in{\bb T}_n}H\Big(\frac{x}{n}\Big)\Big(\eta_{tn^2}(x)-\rho\Big).
\end{equation}
We will use the more compact notation $\bar{\eta}(x)$ to denote $\eta(x)-\rho$.
The equilibrium density fluctuations for these models was proved in Theorem 2.1 of \cite{fsv} and is stated as follows. Denote by $\mc D([0,T], S_W'(\bb T))$ the path space of
c\`adl\`ag trajectories with values in $S_W'(\bb T)$.

 \begin{theorem}\label{t2}

As $n$ goes to infinity, the sequence $\{\mc{Y}_t^n\,:\,t\in{[0,T]} \}_{n\in{\mathbb{N}}}$
converges, in the Skorohod topology of $\mc D([0,T], S_W'(\bb T))$, to $\{\mc{Y}_t\,:\,t\in{[0,T]} \}$ the generalized Ornstein-Uhlenbeck process which is the stationary solution of the stochastic partial differential equation given by
\begin{equation}\label{eqOU}
d\mc Y_t=\tilde{c}\;'(\rho)\mf {L}_W\mc Y_tdt+\sqrt{2\chi(\rho)\tilde{c}\;'(\rho)}d\mc B_t,
\end{equation}
where $\chi(\rho)=\rho(1-\rho)$, $\tilde{c}\;'(\rho)=1+2b\rho$ and $\mc B_t$ is a $S'_W(\bb T)$-valued Brownian motion with quadratic variation given by
\begin{equation*}
\<\mc B(H)\>_t=t\int_{\mathbb T}(\partial_W H(x))^2\,W(dx).
\end{equation*}

\end{theorem}

\section{Proof of Theorem \ref{th2.1}}
The proof of this theorem relies on two steps. First, we claim that the occupation time is close to an additive functional of the density fluctuation field $\mc{Y}_t^n$, this is what we called the Local  Replacement. Second, we use Theorem \ref{t2} to prove that the additive functional of the density fluctuation field $\mc{Y}_t^n$ converges to a Gaussian process.
Before proving these two claims we develop some crucial estimates that we need in due course.

\subsection{The Local Replacement.}

For a function $g\in L^2(\nu_\rho)$, we denote by $\mc D_n(g)$  the Dirichlet form of the function $g$,  defined as:
$\mc D_n (g) \;=\;- \int g(\eta) \mc L_n g(\eta){\nu_{\rho}}(d\eta).$
 An elementary computation shows
that
\begin{equation}\label{dirichlet}
\mc D_n (g) \;=\; \sum_{x\in \bb T_n}\frac{ \xi_{x,x+1}^n}{2}
\int  c_{x,x+1}(\eta)\Big( g(\eta^{x,x+1}) -
g(\eta) \Big)^2 \,\nu_\rho(d\eta).
\end{equation}

 \begin{lemma}[{\em{Local Replacement}}] \label{lemma7.1}
\quad

 For   all $\ell\geq{1}$ and $t\in{[0,T]}$, it holds that
\begin{equation*}
\mathbb{E}_\rho\Big[\Big(\int_{0}^t \{\bar \eta_{sn^2}(0)-\bar
\eta_{sn^2}^\ell(0)\}ds\Big)^2\Big]\leq{\frac{ 20 \,t\,}{n^2\ell} C(\rho)\sum_{y=0}^{\ell-1}\sum_{z=0}^{y-1} \frac{1}{\xi^n_{z,z+1}}
},
\end{equation*}
where
\begin{equation*}
\bar{\eta}^{\ell}_{sn^2}(0)=\frac{1}{\ell}\sum_{y=0}^{\ell-1}\bar{\eta}_{sn^2}(y)
\end{equation*}
and $C(\rho)$ is a positive constant.
\end{lemma}

In order to prove the last lemma, we use the following result.

 \begin{lemma}
For $g\in{L^2(\nu_\rho)}$ and for a constant $A>0$, it holds that
\begin{equation*}
\int_{\Omega_n}\! \{\bar \eta(0)-\bar \eta^\ell(0)\}g(\eta)\nu_\rho(d\eta)\leq{\frac{A}{2\ell} \sum_{y=0}^{\ell-1}\sum_{z=0}^{y-1} \frac{1}{\xi^n_{z,z+1}}\int_{\Omega_n}\frac{1}{c_{z,z+1}(\eta)}\nu_\rho(d\eta)
+ \frac{1}{2A}\mc{D}_n(g)}.
\end{equation*}
\end{lemma}

\begin{proof}
By the definition of the empirical average $\bar{\eta}^\ell(0)$, we can rewrite the integral on the left hand side in the statement of the lemma as
\begin{equation*}
\frac{1}{\ell} \sum_{y=0}^{\ell-1}\sum_{z=0}^{y-1} \int_{\Omega_n}\{\eta(z)-\eta(z+1)\}g(\eta)\nu_\rho(d\eta).
\end{equation*}
Writing the previous expression as twice its half and performing the change of variables $\eta\mapsto\eta^{z,z+1}$, for which the
measure $\nu_\rho$ is invariant, it equals to
\begin{equation*}
\frac{1}{2\ell}\sum_{y=0}^{\ell-1}\sum_{z=0}^{y-1}\int_{\Omega_n} (\eta(z)-\eta(z+1))(g(\eta)-g(\eta^{z,z+1}))\nu_\rho(d\eta).
\end{equation*}
By the Cauchy-Schwarz inequality we bound the expression above by
\begin{equation*}
\begin{split}
&\frac{1}{2\ell} \sum_{y=0}^{\ell-1}\sum_{z=0}^{y-1}\frac{1}{\xi^n_{z,z+1}} \int_{\Omega_n}\frac{A}{c_{z,z+1}(\eta)}(\eta(z)-\eta(z+1))^2\nu_\rho(d\eta)\\+&
\frac{1}{2\ell} \sum_{y=0}^{\ell-1}\sum_{z=0}^{y-1} \xi_{z,z+1}^n\int_{\Omega_n}\frac{c_{z,z+1}(\eta)}{A}(g(\eta)-g(\eta^{z,z+1}))^2\nu_\rho(d\eta).
\end{split}
\end{equation*}
To finish the proof it is enough to recall  \eqref{dirichlet}.
\end{proof}

\begin{proof}[Proof of Lemma  \ref{lemma7.1}.]

By Proposition A1.6.1 of \cite{kl}, we have that
\begin{equation*}
\begin{split}
\mathbb{E}_\rho\Big[&\Big(\int_{0}^t \{\bar \eta_{sn^2}(0)-\bar \eta_{sn^2}^\ell(0)\}ds\Big)^2\Big]\\
&\leq20 \,t\, \sup_{g\in{L^2(\nu_\rho)}}\Big\{2\int_{\Omega_n} \{\bar \eta(0)-\bar \eta^\ell(0)\}g(\eta)\nu_\rho(d\eta) - n^2\mc{D}_n(g)\Big\}\\
&\leq 20 \,t\, \sup_{g\in{L^2(\nu_\rho)}} \Big\{ \frac{A}{\ell} \sum_{y=0}^{\ell-1}\sum_{z=0}^{y-1} \frac{1}{\xi^n_{z,z+1}}\int_{\Omega_n}\frac{1}{c_{z,z+1}(\eta)}\nu_\rho(d\eta)+ \frac{1}{A}\mc{D}_n(g) -n^2\mc{D}_n(g)\Big\}.
 \end{split}
\end{equation*}
In last inequality we used the previous lemma.
Taking $1/A=n^2$ we get the bound
\begin{equation*}
\frac{ 20 \,t\,}{n^2 \ell} \sum_{y=0}^{\ell-1}\sum_{z=0}^{y-1} \frac{1}{\xi^n_{z,z+1}}\int_{\Omega_n}\frac{1}{c_{z,z+1}(\eta)}\nu_\rho(d\eta).
\end{equation*}
To conclude it is enough to observe that
\begin{equation}
\int_{\Omega_n}\frac{1}{c_{z,z+1}(\eta)}\nu_\rho(d\eta)=(1-\rho)^2+\frac{2}{1+b}\rho(1-\rho)+\frac{1}{1+2b}\rho^2\;:=\;C(\rho).
\end{equation}
\end{proof}
 \begin{corollary}
For  any $\eps>0$ and any $t\in{[0,T]}$,  it holds that
\begin{equation*}
\mathbb{E}_\rho\Big[\Big(\int_{0}^t \{\bar \eta_{sn^2}(0)-\bar
\eta_{sn^2}^{\eps n}(0)\}ds\Big)^2\Big]\leq{\frac{ 20 \,t}{\eps n^2}C(\rho)\sum_{y=0}^{\eps n-1}\big(W\big(\pfrac{y}{n}\big)-W(0)\big)},
\end{equation*}
for a positive constant $C(\rho)$.
\end{corollary}
\begin{proof}
This result is a consequence of Lemma \ref{lemma7.1} with $\ell=\varepsilon n$ and the fact that $\xi^n_{x,x+1}=\frac{1}{n(W(\frac{x+1}{n})-W(\frac{x}{n}))}$ so that
\begin{equation*}
\sum_{y=0}^{\eps n-1}\sum_{z=0}^{y-1} \frac{1}{\xi^n_{z,z+1}}\leq {n\sum_{y=0}^{\eps n-1}\big(W\big(\pfrac{y}{n}\big)-W(0)\big)}.
\end{equation*}
\end{proof}

 \begin{corollary} \label{Wcorollary}
 For any $\eps\geq 0$, for  any $W$ satisfying \eqref{Wassumption} for some $\theta>0$ and for any $t\in{[0,T]}$,
it holds that
\begin{equation*}
\mathbb{E}_\rho\Big[\Big(\sqrt{n}\int_{0}^t \{\bar \eta_{sn^2}(0)-\bar
\eta_{sn^2}^{\eps n}(0)\}ds\Big)^2\Big]\,\leq\, 20 \,t \,C(\rho)\,\eps^\theta,
\end{equation*}
for a positive constant $C(\rho)$.
\end{corollary}
\begin{proof}
By the previous corollary, the expectation above is bounded from above by
\begin{equation*}
\frac{ 20 \,t}{\eps n}C(\rho)\sum_{y=0}^{\eps n-1}\big(W\big(\pfrac{y}{n}\big)-W(0)\big),
\end{equation*}
and  by the assumption \eqref{Wassumption} last term is smaller than $20 \,t \,C(\rho)\,\eps^\theta$, where $c$ is a constant.
\end{proof}

At this point we are able to use the Local Replacement in order to prove that the occupation time is close to the additive functional of the density of particles. For that purpose, for $\varepsilon \in (0,1)$ we denote by $\iota_\varepsilon$ the function $y  \mapsto \varepsilon^{-1} \textbf{1}_{[0,\varepsilon]}(y)$. The sequence $\{\iota_\varepsilon\,;\, \varepsilon \in (0,1)\}$ is therefore an {\em approximation of the identity}.

\begin{proposition} \label{Wproposition}
Fix $t>0$.
Then
\begin{equation}\label{Wbound}
\bb E_{\rho}\Big[\Big(\Gamma_n(t)-\int_{0}^t\mc{Y}^n_s(\iota_\varepsilon)ds \Big)^2\Big]\leq Ct\varepsilon^{\theta}.
\end{equation}
\end{proposition}
\begin{proof}
Observe that
\begin{equation*}
\begin{split}
\Gamma_n(t)-\int_{0}^{t}\mc{Y}^n_s(\iota_\varepsilon)ds=&\frac{1}{n^{3/2}}\int_{0}^{tn^2}\bar{\eta}_s(0)ds-\int_{0}^{t}\frac{1}{\varepsilon \sqrt{n}}\sum_{x=0}^{\varepsilon n}\bar{\eta}_{sn^2}(x)ds\\
&=\sqrt {n}\int_{0}^{t}\Big(\bar{\eta}_{sn^2}(0)-\bar{\eta}^{\varepsilon n}_{sn^2}(0)\Big)ds.
\end{split}
\end{equation*}

In the first equality we used the definitions of $\Gamma_n(t)$ and $\mc{Y}^n_s$ given, respectively, in \eqref{ot} and \eqref{fluctfield} and the definition of  $\iota_\varepsilon$ given above. In the second equality, we used a change of variables in the time integral.
Now, it is enough to recall Corollary \ref{Wcorollary} in order to finish the proof.
\end{proof}

\subsection{The approximation in the $S_W(\bb T)$  space}

So far, we were able to show that the occupation time is close to the additive functional of the density of particles evaluated at the function $\iota_\eps$. We would like to invoke  the Theorem 2.1 of \cite{fsv} in order to assure the convergence of the density fluctuation  field $\mc{Y}^n_t$ to some process $\mc{Y}_t$, as $n$ tends to infinity. However,  the function $\iota_\eps$ does not belong to the space of test functions $S_W(\bb T)$,  therefore, we can not apply directly the Theorem 2.1 of \cite{fsv} to $\mc{Y}^n_t(\iota_\eps)$. To overcome this problem, we approximate first the function $\iota_\eps$ by a sequence of functions $\{\iota^k_{\varepsilon}\}_{k\in \bb N}$ in the space of test functions $S_W(\bb T)$ . This is the content of the next lemma.

Denote by ${\bf{1}}_A(u)$ the function that takes the value $1$ if $u\in A$ and $0$ if $u\notin A$.

\begin{lemma}\label{iota}
For fixed $\varepsilon\in(0,1)$, there exists a sequence of functions  $\{\iota^k_{\varepsilon}\}_{k\in{\mathbb N}}$ in the space of test functions $S_W(\bb T)$ converging to $\iota_\varepsilon$ in the $L^2(\bb T)$-norm, as $k$ tends to infinity
\end{lemma}
\begin{proof}
In fact, we are going to approximate the function $\iota_\varepsilon$ by a sequence of functions on the space
$\mf D_W$, which is a subset of $S_W(\bb T)$, as defined in Subsection \ref{sub3.3}.

Define
\begin{equation}\label{iotaepsk}
 \iota_\varepsilon^k(u):=\int_{(0,u]} \Big(\int_0^y g_\varepsilon^k(z)\,dz\Big)W(dy), \quad \forall\,u\in\bb T,
\end{equation}
where the function
 $g_\varepsilon^k(z)\in L^2(\bb T)$ is given by
\begin{equation*}
g_\varepsilon^k(z):=c_\varepsilon^{k,+}g_\varepsilon^{k,+}(z)
 -c_\varepsilon^{k,-}g_\varepsilon^{k,-}(z),\quad\forall\, z\in \bb T,
\end{equation*}
 with
\begin{equation*}
\begin{split}
 & g_\varepsilon^{k,+}(z):=k\Big[\textbf{1}_{(0,1/k]}(z)-\textbf{1}_{(1/k,2/k]}(z)\Big],\quad\forall\, z\in \bb T,\\
 & g_\varepsilon^{k,-}(z):=k\Big[\textbf{1}_{(\varepsilon-1/k,\varepsilon]}(z)-\textbf{1}_{(\varepsilon,\varepsilon+1/k]}(z)\Big],\quad\forall\, z\in \bb T,\\
 & c_\varepsilon^{k,+}:=\frac{1}{\varepsilon}\Bigg(\int_{(0,1]}\Big(\int_0^y g_\varepsilon^{k,+}(z)\,dz\Big)W(dy)\Bigg)^{-1},\\
 & c_\varepsilon^{k,-}:=\frac{1}{\varepsilon}\Bigg(\int_{(0,1]}\Big(\int_0^y g_\varepsilon^{k,-}(z)\,dz\Big)W(dy)\Bigg)^{-1}.
\end{split}
\end{equation*}
We consider  $k\in \bb N$ such that $k>1/\eps$ ,in order that the formulas above make sense.
We claim that
\begin{equation*}
 \int_0^1 g_\varepsilon^k(z)\,dz=0\qquad\mbox{and}\qquad \int_{(0,1]}\Big(\int_0^y g_\varepsilon^k(z)\,dz\Big)W(dy)=0.
\end{equation*}
The  first equality above follows from the fact that
   $\int_0^1 g_\varepsilon^{k,+}(z)dz=\int_0^1 g_\varepsilon^{k,-}(z)dz=0$, which can be easily checked. The second equality follows from a simple computation.

Under the choice  of  $g_\varepsilon^k$, the function $\iota_\varepsilon^k$ defined in \eqref{iotaepsk}  has the following behavior: for $u\in(2/k, \varepsilon-1/k]$,  $\iota_\varepsilon^k(u)$ is equal to $\varepsilon^{-1}$ and for   $u\in( \varepsilon+1/k,1]$,
  $\iota_\varepsilon^k(u)$ vanishes. Therefore, for each $k\in\bb N$, the function  $\iota_\varepsilon^k$ differs from  $\iota_\varepsilon$ only on the set  $(0, 2/k]\cup ( \varepsilon-1/k,\varepsilon+1/k]$.

  Since  $|\iota_\varepsilon^k-\iota_\varepsilon|$ is bounded by a constant that does not depend on $k$, the Dominated Convergence Theorem implies that
  $\iota^k_{\varepsilon}$ converges to $\iota_\varepsilon$ in $L^2(\bb T)$, as $k$ goes to infinity, concluding the proof.
\end{proof}

\subsection{The Gaussian limit}

At this point we have all the needed ingredients in order to prove our main result, namely,  Theorem \ref{th2.1}.
In this subsection, we follow the ideas from the proof of the Theorem 2.9 of \cite{GJ}.

We know that the occupation time is close to the additive functional of the density of particles evaluated on $\iota_\eps$, which in turn can be very well approximated by the additive functional of the density of particles evaluated on a function in the space of test functions $S_W(\bb T)$. At this point, we can take the limit  as $n$ tends to infinity, because the convergence of $\mc Y_t^n(H)$ to $\mc Y_t(H)$, holds for any $H\in S_W(\bb T)$.

Next, we prove that the additive functional associated to $\mc Y_t(\iota_\eps)$ converges, as $\eps$ tends to $0$, to a Gaussian process.
For that purpose, define
\begin{equation}\label{gamma_eps}
 \tilde{\Gamma}_\varepsilon(t)=\int_{0}^t\mc{Y}_s(\iota_\varepsilon)ds,
\end{equation}
where $\mc{Y}_t$ is the Ornstein-Uhlenbeck process given in \eqref{eqOU}.

\begin{remark}\label{remark4.1}
\rm
We point out that definition above is, in principle, not well defined since $\iota_\eps$ does not belong to the space $S_W(\mathbb{T})$. To handle that, it is necessary to look at the limit of Cauchy sequences $\{\mc{Y}^n_t(\iota^k_\eps)\}_{k\in\bb N}$, where $\{\iota^k_\eps\}_{k\in{\mathbb{N}}}$ is given in Lemma \ref{iota}. By the convergence of $\mc{Y}^n_t(\iota^k_\eps)$ towards $\mc{Y}_t(\iota^k_\eps)$ as $n$ goes to infinity, and the fact that $\{\mc{Y}^n_t(\iota^k_\eps)\}_{k\in{\mathbb{N}}}$ is a Cauchy sequence in $k$ (uniformly in
$n$), a diagonal argument leads to a precise definition of $\tilde{\Gamma}_\varepsilon(t)$. This was very well detailed in \cite{fgn3} or \cite{GJ}. To keep the present text short, we ask the reader to accept \eqref{gamma_eps} or to go into the details in these references.
\end{remark}
The next lemma characterizes, for fixed $t$, the dependency of  $\tilde{\Gamma}_\varepsilon(t)$ on $\eps>0$.
\begin{lemma}\label{bound}
For any fixed $t \in{[0,T]}$ and any $\varepsilon>\delta>0$,
\begin{equation}\label{useful2}
\bb E\big[\big(\tilde{\Gamma}_\varepsilon(t)-\tilde{\Gamma}_\delta(t)\big)^2\big] \leq C \varepsilon^{\theta} t,
\end{equation}
where $C>0$ is some constant that does not depend on $\eps$ nor $\delta$.
\end{lemma}
\begin{proof}
 Fix $\varepsilon>\delta>0$.
Repeatedly applying the inequality $(x+y)^2\leq{2(x^2+y^2)}$, we bound the expectation in \eqref{useful2} by four times the sum of
\begin{equation}\label{expc1}
\bb E\Big[\Big(\tilde{\Gamma}_\varepsilon(t)-\int_{0}^t\mc{Y}_s^n(\iota_\varepsilon)ds\Big)^2\Big],
\end{equation}
\begin{equation}\label{expc2}
\bb E\Big[\Big(\Gamma_n(t)-\int_{0}^t\mc{Y}_s^n(\iota_\varepsilon)\Big)^2\Big],
\end{equation}
\begin{equation}\label{expc3}
\bb E\Big[\Big(\Gamma_n(t)-\int_{0}^t\mc{Y}_s^n(\iota_\delta)\Big)^2\Big],
\end{equation}
and
\begin{equation}\label{expc4}
\bb E\Big[\Big(\tilde{\Gamma}_\delta(t)-\int_{0}^t\mc{Y}_s^n(\iota_\delta)\Big)^2\Big].
\end{equation}
The term in \eqref{expc2} can be estimated by using Proposition \ref{Wproposition}, from where we get that
 \begin{equation*}
\bb E\Big[\Big(\Gamma_n(t)-\int_{0}^t\mc{Y}_s^n(\iota_\varepsilon)\Big)^2\Big]\leq C\varepsilon^\theta t.
\end{equation*}
Analogously, for \eqref{expc3}, we have
 \begin{equation*}
\bb E\Big[\Big(\Gamma_n(t)-\int_{0}^t\mc{Y}_s^n(\iota_\delta)\Big)^2\Big]\leq C\delta^\theta t<
C\varepsilon^\theta t.
\end{equation*}
The next step is to guarantee that \eqref{expc1} is arbitrarily small for large $n$. We do the following.  By Lemma \ref{iota} there exists a sequence of functions $\{\iota^k_{\varepsilon}\}_{k\in{\mathbb N}}$ in the space of test functions $S_W(\bb T)$ approximating the function $\iota_\varepsilon$ in the $L^2(\bb T)$-norm, as $k$ tends to infinity. By adding and subtracting terms, we bound \eqref{expc1} by four times the sum of the terms below:
\begin{equation}\label{expectations}
 \begin{split}
    & \bb E\Big[\Big(\tilde{\Gamma}_\eps(t)-\int_{0}^t\mc{Y}_s(\iota_\varepsilon^k)ds\Big)^2\Big], \\
    & \bb E\Big[\Big(\int_{0}^t\mc{Y}_s(\iota^k_\varepsilon)-\mc{Y}^n_s(\iota^k_\varepsilon)ds\Big)^2\Big], \\
    & \bb E\Big[\Big(\int_{0}^t\mc{Y}_s^n(\iota^k_\varepsilon)-\mc{Y}_s^n(\iota_\varepsilon)ds\Big)^2\Big]. \\
 \end{split}
\end{equation}
The first expectation in \eqref{expectations} can be estimated by using the linearity of $\mc{Y}_t$ together with Lemma \ref{Wcovariance} (postponed to the appendix), from where we get that
\begin{equation*}
\begin{split}
\bb E\Big[\Big(\int_{0}^t\mc{Y}_s(\iota_\varepsilon)-\mc{Y}_s(\iota^k_\varepsilon)ds\Big)^2\Big]
=t^2\chi(\rho)\int_{\bb T}(\iota_\varepsilon(u)-\iota_\varepsilon^k(u))^2du.
\end{split}
\end{equation*}
By Lemma \ref{iota}, the right hand side of equality above goes to zero as $k$ goes to infinity.

The second  expectation in \eqref{expectations} goes to zero, as $n$ tends to $\infty$, as a consequence of the Theorem 2.1 of \cite{fsv}.

To bound the third expectation in \eqref{expectations} we apply the Cauchy-Schwarz inequality, leading to
\begin{equation*}
\begin{split}
\bb E\Big[\Big(\int_{0}^t\mc{Y}_s^n(\iota^k_\varepsilon)-\mc{Y}_s^n(\iota_\varepsilon)ds\Big)^2\Big]\leq t^2\frac{1}{n}\sum_{x\in{\bb T_n}}(\iota^k_\varepsilon-\iota_\varepsilon)^2\Big(\frac{x}{n}\Big)\chi(\rho).
\end{split}
\end{equation*}
Taking $n$  sufficiently large, the right hand of the previous expression is close to
$$t^2\int_{\bb T}(\iota_\varepsilon^k(u)-\iota_\varepsilon(u))^2du\, \chi(\rho)$$ and again by Lemma \ref{iota}, this expression is small for $k$ sufficiently big.

Expression \eqref{expc4} can be treated in same way as \eqref{expc1}, finishing the proof of the lemma.
\end{proof}

\begin{proposition}
As $\varepsilon$ goes to zero, the sequence of processes $\{\tilde{\Gamma}_\varepsilon(t):t\in{[0,T]}\}_{\varepsilon>0}$ converges in distribution, with respect to the uniform topology of $\mc{C}([0,T],\bb R)$, to a Gaussian process $\{\tilde \Gamma(t): t\in{[0,T]}\}$.
\end{proposition}

\begin{proof}
We begin by claiming that
\begin{equation}\label{useful}
\bb E\Big[\Big(\tilde{\Gamma}_\varepsilon(t)\Big)^2\Big] \leq
t^2\frac{\chi(\rho)}{\varepsilon}.
\end{equation}
 By the Cauchy-Schwarz inequality,
 \begin{equation*}
 \begin{split}
\bb E\Big[\Big(\tilde{\Gamma}_\varepsilon(t)\Big)^2\Big]  \leq \ t \,\bb E\Big[ \int_0^t (\mc Y_s(\iota_\eps))^{2}ds\Big].\\
		     		      \end{split}
\end{equation*}
By  Fubini's Theorem and Lemma \ref{Wcovariance} we get that
 \begin{equation*}
 \begin{split}
\bb E\Big[\Big(\tilde{\Gamma}_\varepsilon(t)\Big)^2\Big] \leq t \int_0^t\bb E\Big[ (\mc Y_s(\iota_\eps))^{2}ds\Big]={t^2\chi(\rho)\int_{\bb T}(\iota_\varepsilon(u))^2du}=t^2\frac{\chi(\rho)}{\varepsilon},\\
\end{split}
\end{equation*}
proving the claim. We observe that Lemma \ref{Wcovariance} is stated only for functions in the space $S_W(\mathbb{T})$. Nevertheless, an aproximation procedure in $L^2$ as described in the Remark \ref{remark4.1} extends the statement of the Lemma \ref{Wcovariance} for $\iota_\eps$ as well.

Fix $\eps>0$. For $\delta<\eps$, applying \eqref{useful2} and \eqref{useful} we have that
\begin{equation}
\begin{split}
\bb E\Big[\Big(\tilde{\Gamma}_\delta(t)\Big)^2\Big] & \;=\; \bb E\Big[\Big(\tilde{\Gamma}_\delta(t)-\tilde{\Gamma}_\eps(t)+\tilde{\Gamma}_\eps(t)\Big)^2\Big] \\
	& \; \leq \; 2\bb E\Big[\Big(\tilde{\Gamma}_\delta(t)-\tilde{\Gamma}_\eps(t)\Big)^2\Big] + 2
	\bb E\Big[\Big(\tilde{\Gamma}_\eps(t)\Big)^2\Big]\\
	& \; \leq \; 2C \varepsilon^\theta t + 2t^2\frac{\chi(\rho)}{\varepsilon}.
\end{split}
\end{equation}
If $t\geq{\delta^{1+\theta}}$,  taking $\varepsilon = t^{1/{1+\theta}}$ we conclude that
\begin{equation}
 \bb E\Big[\Big(\tilde{\Gamma}_\delta(t)\Big)^2\Big] \leq Ct^{\pfrac{1+2\theta}{1+\theta}},
\end{equation}
where $C$ does not depend on $\varepsilon$ nor $t$. On the other hand, if $t<\delta^{1+\theta}$, then $t^{\pfrac{1}{1+\theta}}<\delta$ and using \eqref{useful}, the previous inequality is also true.
Therefore, by the stationarity of $\mc Y_t$ and since $\tilde{\Gamma}_\varepsilon(0)=0$, we get that
\begin{equation}\label{eq416}
\begin{split}
 \bb E\big[\big(\tilde{\Gamma}_\varepsilon(t)-\tilde{\Gamma}_\varepsilon(s)\big)^2\big]  & =
  \bb E\big[\big(\tilde{\Gamma}_\varepsilon(t-s)-\tilde{\Gamma}_\varepsilon(0)\big)^2\big] \\
  & =   \bb E\big[\big(\tilde{\Gamma}_\varepsilon(t-s)\big)^2\big] \leq C|t-s|^{\pfrac{1+2\theta}{1+\theta}},
\end{split}
\end{equation}
for all $t,s\in{[0,T]}$.  Estimate \eqref{eq416} allows to invoke Kolmogorov-Centsov's compactness criterion (see problem 2.4.11 in \cite{k}), assuring that the sequence of processes
$\{\tilde{\Gamma}_\varepsilon(t):t\in{[0,T]}\}_{\varepsilon>0}$ is tight. Besides that, for fixed $t$,  \eqref{useful2} implies that $\{\tilde{\Gamma}_\varepsilon(t)\}_{\eps>0}$ is a Cauchy sequence in $L^2$, implying the uniqueness of limit points. This concludes the proof.

\end{proof}
A final observation: since we have, in general, no manageable formula for the semigroup $\{P_t:t\geq 0\}$  associated to  $\tilde c'(\rho)\mf {L}_W$, we are not able to explicitly characterize the covariance of the Gaussian process $\{\tilde{\Gamma}(t):t\in{[0,T]}\}$ obtained above (and hence we can not characterize the process itself beyond of proving its existence). In next subsection we detail two cases where the covariances can be computed explicitly.

\section{Examples}

We begin by pointing out that the next two examples are models evolving in infinite volume. Nevertheless, the previous results remains in force since they are local estimates, which can be applied in this context.

\subsection{Porous media models}
In this section we consider a collection of models whose scaling limits were studied in \cite{glt}. First, we consider the Markov process $\{\eta_t\, :\, t\geq{0}\}$ evolving on  $\bb Z$ and we take $W$ as the identity function without requiring \eqref{periodic}, so that $\xi_{x,x+1}=1$ for all $x\in{\mathbb{Z}}$. The generator \eqref{ln} of this process acts on local functions  $f:\{0,1\}^{\bb Z}\rightarrow{\mathbb{R}}$ as
\begin{equation*}
(\mc{L}_{n}f)(\eta)=\sum_{x\in \bb Z}\,(1+b(\eta(x-1)+\eta(x+2)))\,\Big[f(\eta^{x,x+1})-f(\eta)\Big],
\end{equation*}
where $b>-1/2$. Above,
 $\eta^{x,x+1}$ denotes the configuration obtained from $\eta$ by exchanging the occupation variables $\eta(x)$ and $\eta(x+1)$, see \eqref{exchange}.

In the particular case where $b=0$, the process defined by the generator above becomes the well known symmetric simple exclusion process.

In this case all the results stated before are still true for this choice of $W$, therefore Theorem \ref{th2.1} holds and we recognize the limit process as a fractional Brownian motion of Hurst exponent $H=3/4$.


Moreover, still considering $W$ as the identity function, we can take more general rates  $c_{x,x+1}(\eta)$  equal to
\begin{equation*}
1+b\Big(\prod_{j=-(m-1)}^{-1}\!\!\!\!\eta(x+j)+\!\!\!\!\prod_{\substack{j=-(m-2)\\j\neq{0,1}}}^{2}\!\!\!\!\eta(x+j)+
\cdots+\!\!\prod_{\substack{j=-1\\j\neq 0,1}}^{m-1} \eta(x+j) + \!\prod_{j=2}^{m}\eta(x+j)\Big),
\end{equation*}
with $b>-1/2$ and $m\in{{\bb N}\backslash\{1\}}$.
Under these rates, particles more likely hop to unoccupied nearest-neighbor sites when there are at least $m-1\geq 1$ other neighboring sites fully occupied.

In \cite{glt} the equilibrium fluctuations were obtained for these models and the limit is a generalized Ornstein-Uhlenbeck process. Since our Local Replacement is also true for these models, then we are able to show the following result.

\begin{theorem}\label{thmmp}
As $n$ goes to infinity, the sequence of processes $\{\Gamma_n(t):t\in{[0,T]}\}_{n\in{\mathbb{N}}}$ converges in distribution with respect to the uniform topology of $\mc{C}([0,T],\mathbb{R})$ to a fractional Brownian motion of Hurst exponent $H=3/4$.
\end{theorem}

The proof of last theorem relies on the two facts. First, the results of the previous section  are true when considering processes evolving on $\mathbb{Z}$, since our results are local estimates. Second, the equilibrium density fluctuations was proved in \cite{glt}. This is enough to obtain the previous result.

\subsection{Symmetric Exclusion with a slow bond}

In this section, we consider the Markov process $\{\eta_t\, :\, t\geq{0}\}$ evolving on $\bb Z$ with $b=0$ and conductances given by
\begin{equation}\label{eq53}
\xi^n_{x,x+1}\;=\;\left\{\begin{array}{cl}
\frac{\alpha}{n^\beta}, &  \mbox{if}\,\,\,\,x=-1,\\
1, &\mbox{otherwise\,,}
\end{array}
\right.
\end{equation}
for $\alpha>{0}$ and $\beta\in[0,\infty]$.

These models correspond to the symmetric exclusion process with a slow bond, which was extensively studied in \cite{fgn,fgn2,fgn3}. We first notice that if we take the process evolving on $\bb T$, the case $\beta=1$ and $\alpha=1$ is a particular case of the ones described above by simply taking $W$ as the identity function and satisfying \eqref{periodic}. Nevertheless, when we take other values of $\beta$ or $\alpha$ we can only write the conductances in terms of a function $W$ that depends on $n$ and this is not covered by the results we presented above, since there the function $W$ is fixed.

Anyhow, we are able to prove the Theorem  \ref{th2.1} for this process evolving on $\bb Z$ and for all the ranges of the parameters $\alpha>0$ and $\beta\in[0,\infty]$. Consider the generator given on local functions  $f:\{0,1\}^{\bb Z}\rightarrow{\mathbb{R}}$  as
\begin{equation}
(\mc{L}_{n}f)(\eta)=\sum_{x\in \bb Z}\,\xi^n_{x,x+1}\,\Big[f(\eta^{x,x+1})-f(\eta)\Big],
\end{equation}
where $\eta^{x,x+1}$ was defined in \eqref{exchange}.

For these models, the space of test functions is  a space that we denote by  $\mc S_\beta(\bb R)$.
In order to define $\mc S_\beta(\bb R)$, we define first
$\mc S(\bb R\backslash\{0\})$, which is the space of functions
 $H:\bb R\to\bb R$, such that $H\in C^\infty(\bb R\backslash\{0\})$ and $H$ is
continuous from the right at $x=0$, with
\begin{equation*}
 \Vert H \Vert_{k,\ell}\;:=\;\sup_{x\in \bb R\backslash{\{0\}}}|(1+|x|^\ell)
\,H^{(k)}(x)|\;<\;\infty\,,
\end{equation*}
\noindent for all integers $k,\ell\geq 0$ and $H^{(k)}(0^-)=H^{(k)}(0^+)$, for all $k$ integer, $k\geq 1$, where
$$H(0^+):=\displaystyle\lim_{\at{u\to 0,}{u>0}}H(u)\quad \textrm{ and }\quad  H(0^-):=\displaystyle\lim_{\at{u\to
0,}{u<0}}H(u)\,,$$
when the above limits exist. Now, let $\mc S_\beta(\bb R)$
 be the subset of $\mc S(\bb R\backslash\{0\})$ composed of functions $H$
satisfying:
\begin{itemize}
\item for  $\beta\in[0,1)$, $H(0^-)=H(0^+)\,,$
\item  for $\beta=1$,
$H^{(1)}(0^+)\;=\; H^{(1)}(0^-)\;=\; \alpha \Big( H(0^+)-H(0^-)\Big)\,.$
\item for $\beta\in(1,\infty]$, $H^{(1)}(0^+)\;=\; H^{(1)}(0^-)\;=\; 0\,.$
\end{itemize}

Now, we are able to show the following result.

\begin{lemma}
For fixed $\varepsilon\in(0,1)$, there exists a sequence of functions  $\{\iota^k_{\varepsilon}\}_{k\in{\mathbb N}}$ in the space of test functions $\mc S_\beta(\bb R)$ converging to $\iota_\varepsilon$ in the $L^2(\bb T)$-norm, as $k$ tends to infinity.
\end{lemma}
\begin{proof}
This proof is the same proof as in Lemma \ref{iota}, if we consider
\begin{equation*}
 \iota_\varepsilon^k(u):=\int_{-\infty}^u \Big(\int_{-\infty}^y h_\varepsilon^k(z)\,dz\Big)dy, \quad \forall\,u\in\bb R,
\end{equation*}
where $h_\varepsilon^k$ is an approximation of the function $g_\varepsilon^k$, defined above, in the space $\mc S(\bb R\backslash \{0\})$.

Then the function $\iota^k_{\varepsilon}$ belongs to space of test functions  $\mc S_\beta(\bb R) $,  and converges to $\iota_\varepsilon$, as $k$ tends to infinity in the $L^2(\bb T)$-norm.
\end{proof}

\begin{theorem}\label{thsbond}
As $n$ goes to infinity, the sequence of processes $\{\Gamma_n(t):t\in{[0,T]}\}_{n\in{\mathbb{N}}}$ converges in distribution with respect to the uniform of $\mc{C}([0,T],\mathbb{R})$ to:

\quad

 $\bullet$ For $\beta\in[0,1)$, a mean-zero Gaussian process  $\{{\Gamma_\infty}(t):t\in{[0,T]}\}$ with variance given by
\begin{equation}\label{variance fbm infty}
\bb E\big[\big(\Gamma_\infty(t)\big)^2\big] =\frac{4}{3}\frac{\chi(\rho)}{\sqrt \pi}t^{3/2}.
\end{equation}

Or else, $\{{\Gamma_\infty}(t):t\in{[0,T]}\}$ is a fractional Brownian motion of Hurst exponent $3/4$.

\quad

 $\bullet$ For $\beta=1$, a mean-zero Gaussian process  $\{{\Gamma_\alpha}(t):t\in{[0,T]}\}$ with variance given by
\begin{equation}
\bb E\big[\big(\Gamma_\alpha(t)\big)^2\big] =\frac{4}{3}\frac{\chi(\rho)}{\sqrt \pi}t^{3/2}+2\chi(\rho)\int_{0}^t\int_{0}^s
\frac{F_{\alpha}(s-r)}{\sqrt{4\pi (s-r)}} dr ds,
\end{equation}
where
\begin{equation}\label{Falpha}
  F_{\alpha}(t)=\frac{1}{2t}\int_{0}^{+\infty}z\; e^{{-z^2/4t-2\alpha z}}\; dz.
  \end{equation}
Moreover, this process $\{{\Gamma_\alpha}(t):t\in{[0,T]}\}$ is not self-similar, hence it is not a fractional Brownian motion.
 \quad

 $\bullet$ For $\beta\in(1,\infty]$, a mean-zero Gaussian process  $\{{\Gamma_0}(t):t\in{[0,T]}\}$ with variance given by
\begin{equation}\label{variance fbm 0}
\bb E\big[\big(\Gamma_0(t)\big)^2\big] =\frac{8}{3}\frac{\chi(\rho)}{\sqrt \pi}t^{3/2}.
\end{equation}

Or else, $\{{\Gamma_0}(t):t\in{[0,T]}\}$ is a fractional Brownian motion of Hurst exponent $3/4$ with twice the variance of $\{{\Gamma_\infty}(t):t\in{[0,T]}\}$.
\end{theorem}
\begin{proof}
The results follow from  Lemma \ref{lemma7.1}, which is also true for the model evolving on $\bb Z$, since it is basically a local replacement which is being done there. Since $\xi^n_{x,x+1}=1$  we obtain
\begin{equation*}
\mathbb{E}_\rho\Big[\Big(\int_{0}^t \{\bar \eta_{sn^2}(x)-\bar
\eta_{sn^2}^\ell(x)\}ds\Big)^2\Big]\leq{\frac{ 20 \,t\,\ell}{n^2}}.
\end{equation*}

Moreover, by Theorem 2.6 of \cite{fgn3} and Theorem 2.1 of \cite{GJ} the result follows. In order to characterize the limiting processes, by stationarity, since they are mean-zero and equal to $0$ for $t=0$,  it is enough to compute their variances.
By symmetry, we get that

\begin{equation*}
\bb E\big[\big(\Gamma_\alpha(t)\big)^2\big] =\lim_{\varepsilon\to 0}2\int_{0}^{t}\int_{0}^{s}\bb E[\mc{Y}_s(\iota_\varepsilon)\mc{Y}_r(\iota_\varepsilon)]\,dr\, ds,
\end{equation*}
where $\mc Y_t$ is the stationary solution of
\begin{equation}
d\mc{Y}_t=\Delta_\beta \mc{Y}_tdt+\sqrt{2\chi(\rho)} \nabla_\beta d\mc{W}_t,
\end{equation}
 being $\mc{W}_t$ is a space-time white noise of unit variance and the characteristic operators $\Delta_\beta$ and $ \nabla_\beta$ were defined in \cite{fgn3}.
By equation (33) in the proof of Theorem 2.7 of \cite{fgn3},
\begin{equation*}
\bb E[\mc{Y}_s(\iota_\varepsilon)\mc{Y}_r(\iota_\varepsilon)]=\chi(\rho)\int_{\mathbb{R}}(T^\beta_{s-r}\iota_\varepsilon)(u)\iota_\varepsilon(u) du,
\end{equation*}
where $T^\beta_{t}$ is the semigroup associated to $\Delta_\beta$. It  remains only to take the limit of expression above as $\eps$ goes to zero. Performing a simple but long computation we get the result. For the sake of completeness we present this computation in the Lemma \ref{lemma_appendix} of the Appendix.

Finally, the fact that $\{{\Gamma_\alpha}(t):t\in{[0,T]}\}$ is not self-similar it is a consequence of the fact that its variance is not invariant under a time-transformation of a power type,  see \cite{l}.
\end{proof}

It is a folklore conjecture that the fluctuations of current and occupation times should be of same order.
By means of the previous theorem, we offer a counter-example for such idea. In \cite{fgn3}, it was proved that  the fluctuations for the current at the origin in the regime $\beta>1$ are null. Opposed to that, in the theorem above, we get that the fluctuations for the occupation times at the origin are not null. Of course, this does not eliminate the possibility the conjecture to be true under some additional hypothesis on the particle system. Anyway, the particle system we have used here to present the counter-example has some strong properties as reversibility and order preservation of the particles.

As a consequence of the Theorem \ref{thsbond} we discover also that the three processes obtained as the limit of the occupation time are continuously related through the parameter $\alpha$ presented in \eqref{eq53}. This result is stated in the following corollary.

\begin{corollary} \label{limit robin ot}
The sequence of processes $\{{\Gamma}_\alpha({t}):t\in{[0,T]}\}_{\alpha>0}$ converges, as $\alpha$ tends to infinity, to the
mean-zero Gaussian process $\{{\Gamma_\infty}(t):t\in{[0,T]}\}$
with variance given by \eqref{variance fbm infty}. On the other hand, as $\alpha$ tends to zero,
the sequence of processes $\{{\Gamma}_\alpha({t}):t\in{[0,T]}\}_{\alpha>0}$ converges
to the mean zero Gaussian process $\{{\Gamma_0}(t):t\in{[0,T]}\}$
with variance given by \eqref{variance fbm 0}.
The convergence above is in the sense of finite dimensional distributions.
 \end{corollary}

 \begin{proof}
Gaussian processes are characterized by their covariance. Reversibility in all cases allows to characterize the covariance in terms of the variance. Therefore, it is enough to show the convergence of the variances in each case. In its hand, this is a consequence of the Dominated Convergence Theorem and the fact that
\[\forall\, t\geq 0, \;\lim_{\alpha\rightarrow{\infty}}F_{\alpha}(t)=0 \;\textrm{ and }\; \lim_{\alpha\rightarrow{0}}F_\alpha(t)=1,\]
where $F_\alpha(t)$ was defined in \eqref{Falpha}. Both limits above are of straightforward verification.
 \end{proof}

\section{Appendix}

We present in this appendix the proof of the following lemma, which it is a standard one in the area. Because we were not able to find it written anywhere in the literature, we include it here for the sake of completeness.
\begin{lemma}\label{Wcovariance}
If $\{\mc{Y}_t:t\geq{0}\}$ is a solution of \eqref{eqOU}, then for all $H\in S_W(\bb T)$, it holds that
\begin{equation}
\label{varianceflu}
 \bb E\big[\mc Y_t(H)\mc Y_s(H)\big] =\chi(\rho) \int_{\bb T}(P_{t-s}H)(u)H(u)du,
\end{equation}
where $\{P_t:t\geq 0\}$ is the semigroup associated to $\tilde c'(\rho)\mf {L}_W$.
\end{lemma}

\begin{proof}
From \cite{fsv}, since $\mc Y_t$ solves \eqref{eqOU}, then $\mc Y_t$ solves the following martingale problem: for every $H\in{S_W(\bb T) }$,
\begin{equation} \label{mart}
\mc M_t(H)=\mc Y_t(H)-\mc Y_0(H)-\tilde{c}'(\rho)\int_{0}^{t}\mc Y_s(\mc L_W H)ds
\end{equation}
is a martingale with respect to the natural filtration $\mc{F}_t:=\sigma(\eta_s: 0\leq s\leq t)$.
At first, we claim that
\begin{equation}
 \bb E\big[\mc Y_t(H)\mc Y_0(H)\big] =\chi(\rho) \int_{\bb T}(P_{t}H)(u)H(u)du.
\end{equation}
For this purpose, notice that
\begin{equation}\label{eq6.4a}
\begin{split}
 &\bb E\big[\mc Y_t(H)\mc Y_0(H)\big]  = \bb E\Big[\Big(\mc M_t(H)+\mc Y_0(H)+\tilde{c}'(\rho)\int_{0}^{t}\mc Y_s(\mc L_W H)ds\Big) \, \mc Y_0(H) \Big]\\
 &=\bb E\big[\mc M_t(H)\mc Y_0(H)\big]+\bb E\big[\mc Y_0(H)\mc Y_0(H)\big]+\bb E\Big[\mc Y_0(H)\;\tilde{c}'(\rho)\int_{0}^{t}\mc Y_s(\mc L_W H)ds
\Big].
\end{split}
\end{equation}
The first expectation above vanishes because
\begin{equation*}
\begin{split}
\bb E\big[\mc M_t(H)\mc Y_0(H)\big]=\bb E\big[\bb E\big[\mc M_t(H)\mc Y_0(H)\big|\mc F_0\big]\big]&=\bb E\big[\mc Y_0(H)\bb E\big[\mc M_t(H)\big|\mc F_0\big]\big]\\
&=\bb E\big[\mc Y_0(H)\mc M_0(H)\big]=0,\\
\end{split}
\end{equation*}
where last equality above is due to $\mc M_0(H)=0$.

The second term can be handled as follows. By computing the characteristic function of $\mc Y^n_0(H)$ and by the Theorem 2.1 of \cite{fsv}, we get that
\begin{equation}\label{eq6.41}
\bb E\big[\mc Y_0(H)\mc Y_0(H)\big]=\chi(\rho)\int_{\bb T}(H(u))^2du.
\end{equation}
Now we develop the last expectation of \eqref{eq6.4a} by using again \eqref{mart}, that is:
\begin{equation*}
\begin{split}
&\bb E\Big[\tilde{c}'(\rho)\int_{0}^{t}\mc Y_s(\mc L_W H)ds\mc Y_0(H)\Big]=\int_{0}^{t}\bb E\Big[\mc Y_s\Big(\tilde{c}'(\rho)\mc L_W H\Big)\mc Y_0(H)\Big]\; ds\\
&=\int_{0}^{t}\bb E\Big[\big(\mc M_s(\tilde{c}'(\rho)\mc L_W H)+\mc Y_0(\tilde{c}'(\rho)\mc L_W H)+\mc Y_0(H)\int_{0}^{s}\mc Y_r\Big((\tilde{c}'(\rho))^2\mc L^2_W H \Big)dr\Big]ds.
\end{split}
\end{equation*}
Repeating the same argument as above, last expression can be rewritten as
\begin{equation*}
\chi(\rho)\int_{\bb T}t\tilde{c}'(\rho)(\mc L_W H)(u)H(u)du+\int_{0}^{t}\int_{0}^{s}\bb E\Big[\mc Y_r\Big((\tilde{c}'(\rho))^2\mc L^2_W H\Big)\mc Y_0(H)\Big]\;dr\; ds.
\end{equation*}
Let us introduce the temporary notation $G:=(\tilde{c}'(\rho))^2\mc L^2_W H$ and rewrite expression above simply as
\begin{equation}\label{eq6.42}
\chi(\rho)\int_{\bb T}t\tilde{c}'(\rho)(\mc L_W H)(u)H(u)du+\int_{0}^{t}\int_{0}^{s}\bb E\big[\mc Y_r(G)\mc Y_0(H)\big]\;dr\; ds.
\end{equation}
We want to characterize the expectation in the second parcel above.
Invoking \eqref{mart} again we have that
\begin{equation*}
\mc M_r(G)=\mc Y_r(G)-\mc Y_0(G)-\tilde{c}'(\rho)\int_{0}^{r}\mc Y_l(\mc L_W G)dl
\end{equation*}
is a martingale. Provided by this fact and repeating the previous arguments, we are lead to
\begin{equation}\label{eq6.4}
 \bb E\big[\mc Y_r(G)\mc Y_0(H)\big]=\chi(\rho)\int_{\bb T}G(u)H(u)du+\bb E\Big[\mc Y_0(H)\;\tilde{c}'(\rho)\int_{0}^{r}\mc Y_l(\mc L_W G)dl\Big]
\end{equation}
Putting together \eqref{eq6.41}, \eqref{eq6.4} and \eqref{eq6.42}, we obtain:
\begin{equation*}
\begin{split}
 \bb E\big[\mc Y_t(H)\mc Y_0(H)\big]& =\chi(\rho)\int_{\bb T}(H(u))^2du+\chi(\rho)t\int_{\bb T}\tilde{c}'(\rho)(\mc L_W H)(u)H(u)du\\&+
\chi(\rho)\frac{t^2}{2}\int_{\bb T}(\tilde{c}'(\rho))^2(\mc L^2_W H)(u)H(u)du+M_t(H),
\end{split}
\end{equation*}
where
\begin{equation*}
M_t(H):= \int_{0}^{t}\int_{0}^{s}\bb E\Big[\mc Y_0(H)\;\tilde{c}'(\rho)\int_{0}^{r}\mc Y_l(\mc L_W G)dl\Big]dr\,ds.
\end{equation*}
From the Lemma 3.5 of \cite{fsv}, we have that  $\mc L_W:S_W(\bb T)\to S_W(\bb T)$ is a bounded operator with respect to the norm associated to the inner product defined in \eqref{inner_product}. Therefore, it makes sense to define the exponential of this operator.
A long, but elementary induction argument over the previous formula  leads to
\begin{equation*}
\begin{split}
 \bb E\big[\mc Y_t(H)\mc Y_0(H)\big]& =\chi(\rho)\int_{\bb T}(e^{t\tilde{c}'(\rho)\mc L_W}H)(u)H(u)du\\
 &=\chi(\rho)\int_{\bb T}(P_tH)(u)H(u)du,
\end{split}
\end{equation*}
where $\{P_t:t\geq 0\}$ is the semigroup associated to $\tilde c'(\rho)\mf {L}_W$.
Finally, since  $\{\mc Y_t: t\geq 0\}$ is a stationary process, we get that
\begin{equation*}
 \bb E\big[\mc Y_t(H)\mc Y_s(H)\big] = \bb E\big[\mc Y_{t-s}(H)\mc Y_0(H)\big] =\chi(\rho) \int_{\bb T}(P_{t-s}H)(u)H(u)du,
\end{equation*}
as desired.
\end{proof}

We finish this Appendix fulfilling some details in the proof of Theorem \ref{thsbond}.
Let $T^\beta_t$ be the semigroup associated to the operator $\Delta_\beta$. For $\beta\in [0,1)$ it is the
semigroup related to the heat equation in the line. Quite classical, it acts on $g\in\mc S_\beta(\bb R)$ as
  \begin{equation}\label{sem heat eq}
  T_t g(x)\;=\; \frac{1}{\sqrt{4\pi t}}\int_{\bb R} e^{-\frac{(x-y)^2}{4t}}g(y)\,dy\,,\quad \textrm{for }x\in\bb R.
 \end{equation}
For $\beta\in (1,\infty]$, the semigroup $T_t^\beta$ is also known and it acts on $g\in\mc S_\beta(\bb R)$ as
  \begin{equation}\label{sem heat eq neu}
  T_t^\textrm{Neu} g(x)\;=\;
  \begin{cases}
\displaystyle   \frac{1}{\sqrt{4\pi t}}\int_{0}^{+\infty}\Big[
e^{-\frac{(x-y)^2}{4t}}+e^{-\frac{(x+y)^2}{4t}}\Big]g(y)\,dy\,,\quad &\textrm{for }x>0\,,\\
\displaystyle  \frac{1}{\sqrt{4\pi t}}\int_{0}^{+\infty}\Big[
e^{-\frac{(x-y)^2}{4t}}+e^{-\frac{(x+y)^2}{4t}}\Big]g(-y)\,dy\,,\quad &\textrm{for }x<0.\\
\end{cases}
 \end{equation}

Denote by $g_{\textrm{even}}$ and $g_{\textrm{odd}}$ the even and odd parts of a function $g:\bb R\to \bb R$, respectively, or
else, for $x\in{\mathbb{R}}$,
\begin{equation}\label{odd_even}
 g_{\textrm{even}}(x)=\frac{g(x)+g(-x)}{2}\quad \textrm{and} \quad g_{\textrm{odd}}(x)=\frac{g(x)-g(-x)}{2}.
\end{equation}

As proved in \cite{fgn3}, for $\beta=1$, the semigroup $T_t^\beta$ acts on $g\in\mc S_\beta(\bb R)$   as
  \begin{equation}\label{beta_1}
  \begin{split}
  & T_t^\alpha g(x)= \frac{1}{\sqrt{4\pi t}}\Bigg\{\int_{\bb R}
e^{-\frac{(x-y)^2}{4t}} g_{\textrm{{\rm even}}}(y)\,dy \\
    & + e^{2\alpha x}\int_x^{+\infty} e^{-2\alpha z} \int_0^{+\infty}
\Big[(\pfrac{z-y+4\alpha t}{2t})e^{-\frac{(z-y)^2}{4t}}+(\pfrac{z+y-4\alpha t}{2t})e^{-\frac{(z+y)^2}{4t}}\Big]\,
g_{\textrm{{\rm odd}}}(y)\, dy\, dz\,\Bigg\}\,,\\
  \end{split}
  \end{equation}
\noindent for $x>0$ and
  \begin{equation*}
  \begin{split}
 & T_t^{\alpha} g(x)= \frac{1}{\sqrt{4\pi t}}\Bigg\{\int_{\bb R}
e^{-\frac{(x-y)^2}{4t}} g_{\textrm{{\rm even}}}(y)\,dy \\
    & - e^{-2\alpha x}\int_{-x}^{+\infty} e^{-2\alpha z} \int_0^{+\infty}
\Big[(\pfrac{z-y+4\alpha t}{2t})e^{-\frac{(z-y)^2}{4t}}+(\pfrac{z+y-4\alpha t}{2t})e^{-\frac{(z+y)^2}{4t}}\Big]\,
g_{\textrm{{\rm odd}}}(y)\, dy\, dz\,\Bigg\}\,,\\
  \end{split}
  \end{equation*}
\noindent for $x<0$.

Below, we state and prove a lemma required in the proof of the Theorem \ref{thsbond}.

\begin{lemma}\label{lemma_appendix}
For $\beta\in[0,1)$,
\begin{equation*}
 \lim_{\eps\searrow 0} \int_{\mathbb{R}}(T^\beta_{t}\iota_\varepsilon)(u)\iota_\varepsilon(u) du\; =\; \frac{1}{\sqrt{ 4\pi t}}.
\end{equation*}
For $\beta=1$,
\begin{equation}\label{eq_beta_1}
 \lim_{\eps\searrow 0} \int_{\mathbb{R}}(T^\beta_{t}\iota_\varepsilon)(u)\iota_\varepsilon(u) du \;=\;
 \frac{1}{\sqrt{4\pi t}}\Big(1+\frac{1}{2 t}\int_0^{+\infty} z e^{-\frac{z^2}{4t}-2\alpha z}dz\Big).
\end{equation}
Finally, for $\beta\in(1,\infty]$,
\begin{equation*}
 \lim_{\eps\searrow 0} \int_{\mathbb{R}}(T^\beta_{t}\iota_\varepsilon)(u)\iota_\varepsilon(u) du\; = \;\frac{2}{\sqrt{4\pi t}}.
\end{equation*}
\end{lemma}
\begin{proof}
 Consider $\beta\in[0,1)$. In this case,
 \begin{equation}\label{eq.610}
 \begin{split}
  \lim_{\eps\searrow 0} \int_{\mathbb{R}}(T^\beta_{t}\iota_\varepsilon)(u)\iota_\varepsilon(u) du\; &=\;
  \lim_{\eps\searrow 0}\frac{1}{\eps^2}\int_0^\eps\int_0^\eps \frac{e^{-\frac{(x-y)^2}{4t}}}{\sqrt{4\pi t}}dx\,dy\;=\; \frac{1}{\sqrt{ 4\pi t}}\,,\\
  \end{split}
 \end{equation}
 because the gaussian kernel is a continuous function.
 The case $\beta\in(1,\infty]$ is quite similar. Indeed, for this regime of $\beta$,
\begin{equation*}
 \begin{split}
  \lim_{\eps\searrow 0} \int_{\mathbb{R}}(T^\beta_{t}\iota_\varepsilon)(u)\iota_\varepsilon(u) du\; &=\;
  \lim_{\eps\searrow 0}\frac{1}{\eps^2}\int_0^\eps\int_0^\eps \frac{e^{-\frac{(x-y)^2}{4t}}+e^{-\frac{(x+y)^2}{4t}}}{\sqrt{4\pi t}}dx\,dy\;=\; \frac{2}{\sqrt{ 4\pi t}}\,,\\
  \end{split}
 \end{equation*}
 The case $\beta=1$ deserves more attention.  For $g(u)=\iota_\eps(u)$, we have that
 \begin{equation*}
  g_{\textrm{even}}(u)= \pfrac{1}{2\eps}{\bf 1}_{[-\eps,\eps]}\quad \textrm{ and }\quad g_{\textrm{odd}}(u)= \pfrac{1}{2\eps}\Big({\bf 1}_{(0,\eps]}-{\bf 1}_{[-\eps,0)]}\Big),
 \end{equation*}
according to  \eqref{odd_even}. Recalling formula \eqref{beta_1}, we obtain
\begin{equation}\label{eq6.10}
\int_{\mathbb{R}}(T^\beta_{t}\iota_\varepsilon)(u)\iota_\varepsilon(u) du= \frac{1}{\sqrt{4\pi t}}\Big(
\frac{1}{\eps^2}\int_0^\eps\int_0^\eps e^{-\frac{(x-y)^2}{4t}}\,dy\,dx + \pfrac{S(\eps)}{\eps^ 2}\Big)\,,
\end{equation}
where $S(\eps)$ is
\begin{equation*}
\int_0^\eps\Big( \pfrac{e^{2\alpha x}}{2}\int_x^{+\infty}e^{-2\alpha z}\int_0^\eps
\Big[ \pfrac{(z-y+4\alpha t)}{2t}e^{\frac{-(z-y)^2}{4t}}+\pfrac{(z+y+4\alpha t)}{2t}e^{\frac{-(z+y)^2}{4t}}\Big]\,dy\,dz\Big)dx.
\end{equation*}
We want to precise the limit of \eqref{eq6.10} as $\eps\searrow 0$. By \eqref{eq.610}, it only remains to evaluate the limit of
$S(\eps)/\eps^2$ as $\eps$ goes to zero. A direct verification shows that	
\begin{equation*}
\lim_{\eps\searrow 0} \frac{1}{\eps^2}S(\eps)=\frac{1}{2t}\int_0^\infty ze^{-\frac{z^2}{4t}-2\alpha z}dz,
\end{equation*}
leading to \eqref{eq_beta_1} and hence finishing the proof.
\end{proof}

\section*{Acknowledgements}

The authors thank hospitality to CMAT (Portugal) where this work was initiated, IMPA and PUC (Rio de Janeiro) where it was finished.

T.F. was supported through a grant "BOLSISTA DA CAPES - Bras\'ilia/Brasil" provided by CAPES (Brazil) and a  project PRODOC-UFBA.

A.N. thanks CNPq (Brasil) for support through
the research project "Mec\^anica estat\'istica fora do equil\'ibrio para sistemas estoc\'asticos" Universal n. 479514/2011-9.

P.G. thanks FCT (Portugal) for support through the research
project ``Non - Equilibrium Statistical Physics" PTDC/MAT/109844/2009.
P.G. thanks the Research Centre of Mathematics of the University of
Minho, for the financial support provided by ``FEDER" through the
``Programa Operacional Factores de Competitividade  COMPETE" and by
FCT through the research project PEst-C/MAT/UI00 13/2011.


\begin{thebibliography}{10}


\bibitem{fsv} Farfan, J.;  Simas, A.B;  Valentim, F. J.: {\em Equilibrium fluctuations for exclusion processes with conductances in random environments}. Stochastic Process. Appl., 120, no. 8, 1535--1562, (2010).

\bibitem{fgn}
Franco, T.; Gon\c calves, P.; Neumann, A.:
{\em Hydrodynamical behavior of symmetric exclusion with slow bonds}, Annales de l'Institut Henri Poincar\'e: Probability and Statistics, Volume 49, Number 2, 402--427.

    \bibitem{fgn2} Franco, T.; Gon\c calves, P.; Neumann, A.:
 {\em Phase Transition of a Heat Equation with Robin's Boundary Conditions and Exclusion Process}, arXiv:1210.3662, to appear in Transactions of the American Mathematical Society.

 \bibitem{fgn3} Franco, T.; Gon\c calves, P.; Neumann, A.:
  {\em  Phase transition in equilibrium fluctuations of symmetric slowed exclusion}, Stochastic Processes and their Applications, 123, Issue 12, 4156--4185, (2013).


\bibitem{fl} Franco, T.; Landim, C.: {\em Hydrodynamic Limit of Gradient Exclusion Processes with conductances}. Arch. Ration. Mech. Anal., 195, no. 2, 409--439, (2010).

\bibitem{f} U. Freiberg {\em Analytical properties of measure
    geometric Krein-Feller-operators on the real line} Math.  Nachr.
  {\bf 260} 34--47, (2003).

\bibitem{GJ}
Gon\c calves, P.; Jara, M.: {\em Scaling limits of additive functionals of interacting particle systems},  Communications on Pure and Applied Mathematics, Volume 66, Issue 5, 649--677 (2013).


\bibitem{glt}
Gon\c calves, P.; Landim, C.; Toninelli, C.: {\em Hydrodynamic Limit for a Particle System with degenerate rates}, Annales de l'Institut Henri Poincar\'e: Probability and Statistics, Volume 45, n4, 887--909 (2009).

\bibitem{k}
Karatzas, I.; Shreve, S.: {\em Brownian motion and stochastic calculus}. Graduate Texts in Mathematics, 113. Springer, New York, 1991.


\bibitem{l}  Lamperti, J.W.: {\em Semi-stable processes}. Transactions of the American Mathematical Society, \textbf{104}, 1962.


%



\bibitem{kl} Kipnis, C.; Landim, C.: {\em Scaling limits of interacting
  particle systems}. Grundlehren der Mathematischen Wissenschaften
  [Fundamental Principles of Mathematical Sciences], 320.
  Springer-Verlag, Berlin, (1999).


\bibitem{S}
Sethuraman, S.: {\em Central Limit Theorems for Additive Functionals of the Simple Exclusion Process}, Ann. Probab., \textbf{28},
277--302; Correction (2006) {\bf 34}, 427--428 (2000).


\bibitem{SX}
Sethuraman, S.; Xu, L.: {\em A central limit theorem for reversible exclusion and zero-range particle systems}, Ann. Probab., \textbf{24}, 1842--1870  (1996).





\end{thebibliography}
\end{document}